\documentclass[english,11pt]{article}
\usepackage{ucs}
\usepackage[utf8x]{inputenc}
\usepackage{babel}
\usepackage{amsmath,amsthm}
\usepackage{amssymb}
\usepackage[all]{xy}
\usepackage{geometry}
\usepackage{verbatim}

\newcommand{\C}{\mathbb{C}}
\newcommand{\Q}{\mathbb{Q}}
\newcommand{\Z}{\mathbb{Z}}

\newcommand{\ra}{\rightarrow}

\newtheorem{thm}{Theorem}

\newtheorem{prop}[thm]{Proposition}
\newtheorem{lem}[thm]{Lemma}
\newtheorem{cor}[thm]{Corollary}
\newtheorem{conj}[thm]{Conjecture}

\theoremstyle{remark}
\newtheorem{rk}[thm]{Remark}

\title{Remarks on the Lefschetz standard conjecture and hyperkähler varieties}
\author{François Charles\\
DMA\\
\'Ecole Normale Sup\'erieure\\
45, rue d'Ulm\\
75230 Paris Cedex 05\\
FRANCE\\
francois.charles@ens.fr
}
\date{}
\begin{document}

\maketitle

\paragraph{Abstract.}
We study the Lefschetz standard conjecture on a smooth complex projective variety $X$. In degree $2$, we reduce it to a local statement concerning local deformations of vector bundles on $X$. When $X$ is hyperk\"ahler, we show that the existence of nontrivial deformations of stable hyperholomorphic bundles implies the Lefschetz standard conjecture in codimension $2$.

\section{Introduction}

In the fundamental paper \cite{Gr69} of 1968, Grothendieck states a series of conjectures concerning the existence of certain algebraic cycles on smooth projective algebraic varieties over an algebraically closed ground fields. Those are known as the standard conjectures. In particular, given such a variety $X$ of dimension $n$, the Lefschetz standard conjecture predicts the existence of self-correspondences on $X$ that give an inverse to the operations
$$H^{k}(X) \ra H^{2n-k}(X)$$
given by the cup-product $n-k$ times with a hyperplane section for all $k\leq n$. Here $H^*(X)$ stands for any Weil cohomology theory on $X$, e.g. singular cohomology if $X$ is defined over $\C$, or $l$-adic \'etale cohomology in characteristic different from $l$. If we can invert the morphism $H^{k}(X) \ra H^{2n-k}(X)$ using self-correspondences on $X$, we say that the Lefschetz conjecture holds in degree $k$.

Let us now, and for the rest of the paper, work over $\C$. The Lefschetz standard conjecture then implies the other ones and has strong theoretical consequences. For instance, it implies that numerical and homological equivalence coincide, and that the category of pure motives for homological equivalence is semisimple. We refer to \cite{Kl68} and \cite{Kl91} for more detailed discussions. The Lefschetz standard conjecture for varieties which are fibered in abelian varieties over a smooth curve also implies the Hodge conjecture for abelian varieties as shown by Yves Andr\'e in \cite{An96}. Grothendieck actually writes in the aforementioned paper that ``alongside the resolution of singularities, the proof of the standard conjectures seems to [him] to be the most urgent task in algebraic geometry''.

Though the motivic picture has tremendously developed since Grothendieck's statement of the standard conjectures, very little progress has been made in their direction. The Lefschetz standard conjecture is known for abelian varieties, see \cite{Kl68} and in degree $1$ where it reduces to the Hodge conjecture for divisors. Aside from examples obtained by taking products and hyperplane sections, those seem to be the only two cases where a proof is known.

\bigskip

In this paper, we want to investigate further the geometrical content of the Lefschetz standard conjecture, and try to give insight into the specific case of hyperk\"ahler varieties. The original form of the Lefschetz standard conjecture for a variety $X$ predicts the existence of specific algebraic cycles in the product $X\times X$. Those cycles can be considered as family of cycles on $X$ parametrized by $X$ itself. Our first remark is that the conjecture actually reduces to a general statement about the existence of large families of algebraic cycles in $X$ parametrized by any smooth quasi-projective base. For this, we use Hodge theory on $X$.

It turns out that for those families to give a positive answer to the conjecture, it is enough to control the local variation of the family of cycles considered. Let us give a precise statement. Let $X$ be a smooth projective variety, $S$ a smooth quasi-projective variety, and let $Z\in CH^k(X\times S)$ be a family of codimension $k$ cycles in $X$ parametrized by $S$. Let $\mathcal T_S$ be the tangent sheaf of $S$. Using the Leray spectral sequence for the projection onto $S$ and constructions from Griffiths and Voisin in \cite{IVHS3}, \cite{Vo88}, we construct a map
$$\phi_Z : \bigwedge^k \mathcal T_S \ra H^k(X, \mathcal O_X)\otimes \mathcal O_S,$$
We then get the following result, which we state here only in degree $2$ for simplicity, but see section 2.
\begin{thm}
 Let $X$ be a smooth projective variety.
Then the Lefschetz conjecture is true in degree $2$ for $X$ if and only if there exists a smooth quasi-projective variety $S$, a codimension $2$ cycle $Z$ in $CH^2(X\times S)$ and a point $s\in S$ such that the morphism

$$\phi_{Z,s} : \bigwedge^2 \mathcal T_{S,s} \ra H^2(X, \mathcal O_X)$$
considered above for $k=2$, is surjective.
\end{thm}
This variational approach to the existence of algebraic cycles can be compared to the study of semi-regularity maps as in \cite{Bl72}.

In the following section, we give an explicit formula for $\phi_Z$ in case the cycle $Z$ is given by the Chern classes of a family of vector bundles $\mathcal E$ on $X\times S$. In this situation, we show that $\phi_Z$ is expressed very simply in terms of the Kodaira-Spencer map. Indeed, $\mathcal T_{S,s}$ maps to the space $\mathrm{Ext}^1(\mathcal{E}_s, \mathcal{E}_s)$. We then have a Yoneda product
$$\mathrm{Ext}^1(\mathcal{E}_s, \mathcal{E}_s)\times \mathrm{Ext}^1(\mathcal{E}_s, \mathcal{E}_s)\ra \mathrm{Ext}^2(\mathcal{E}_s, \mathcal{E}_s)$$
and a trace map
$$\mathrm{Ext}^2(\mathcal{E}_s, \mathcal{E}_s)\ra H^2(X,\mathcal O_X).$$
We show that we can express $\phi_{Z,s}$ in terms of the composition
$$\phi_2(\mathcal E) : \bigwedge^2 \mathcal T_{S,s} \ra H^2(X, \mathcal O_X)$$ of those two maps, and we get the following theorem.

\begin{thm}
Let $X$ be a smooth projective variety. Then the Lefschetz conjecture is true in degree $2$ for $X$ if there exists a smooth quasi-projective variety $S$, a vector bundle $\mathcal E$ over $X\times S$, and a point $s\in S$ such that the morphism
\begin{equation}\label{deuxf}
\phi_2(\mathcal E)_s : \bigwedge^2 \mathcal T_{S,s} \ra H^2(X, \mathcal O_X)
\end{equation}
induced by the composition above is surjective.
\end{thm}
The main interest of this theorem is that it makes it possible to only use first-order computations to check the Lefschetz standard conjecture, which is global in nature, thus replacing it by a local statement on deformations of $\mathcal E$. Of course, when one wants to ensure that there exists a vector bundle over $X$ that has a positive-dimensional family of deformations, the computation of obstructions is needed, which involves higher-order computations. However, once a family of vector bundles is given, checking the surjectivity condition of the theorem involves only first-order deformations.

\bigskip

The last part of the paper is devoted to applications of the previous results to hyperk\"ahler varieties. We will recall general properties of those and their hyperholomorphic bundles in section 4. Those varieties have $h^{2,0}=1$, which makes the last criterion easier to check. In the case of $2$-dimensional hyperk\"ahler varieties, that is, in the case of K3 surfaces, Mukai has investigated in \cite{Mu84} the $2$-form on the moduli space of some stable sheaves given by (\ref{deuxf}) and showed that it is nondegenerate. In particular, it is nonzero. Of course, the case of surfaces is irrelevant in our work, but we will use Verbitsky's theory of hyperholomorphic bundles on hyperholomorphic varieties as presented in \cite{Ver96}. In his work, Verbitsky extends the work of Mukai to higher dimensions and shows results implying the nondegeneracy of the the form (\ref{deuxf}) in some cases. Using those, we are able to show that the existence of nonrigid hyperholomorphic bundles on a hyperk\"ahler variety is enough to prove the Lefschetz standard conjecture in degree $2$. Indeed, we get the following.

\begin{thm}\label{nd}
Let $X$ be a projective irreducible hyperk\"ahler variety, and let $\mathcal E$ be a stable hyperholomorphic bundle on $X$. Assume that $\mathcal E$ has a nontrivial positive-dimensional family of deformations. Then the Lefschetz conjecture is true in degree $2$ for $X$.
\end{thm}

In a slightly different direction, recall that the only known hyperk\"ahler varieties, except in dimension $6$ and $10$, are the two families constructed by Beauville in \cite{Be83} which are the small deformations of Hilbert schemes of points on a K3 surface or of generalized Kummer varieties. For those, the Lefschetz standard conjecture is easy -- see \cite{Ar06} for a general discussion -- as their cohomology comes from that of a surface. We get the following.

\begin{thm}\label{stdef}
Let $n$ be a positive integer. Assume that for every K3 surface $S$, there exists a stable hyperholomorphic sheaf $\mathcal E$ with a nontrivial positive-dimensional family of deformations on the Hilbert scheme $S^{[n]}$ parametrizing subschemes of $S$ of length $n$. Then the Lefschetz conjecture is true in degree $2$ for any projective deformation of $S^{[n]}$. The same result holds for generalized Kummer varieties.
\end{thm}
Both those results could be applied taking $\mathcal E$ to be the tangent sheaf of the variety considered, in case it has nontrivial deformations.

Those results fit well in the -- mostly conjectural -- work of Verbitsky as exposed in \cite{Ver99} predicting the existence of large moduli spaces of hyperholomorphic bundles. Unfortunately, we were not able to exhibit bundles satisfying the hypotheses of the theorems.

\bigskip

Varieties are defined to be reduced and irreducible. All varieties and schemes are over the field of complex numbers.

\paragraph{Acknowledgements.}It is a great pleasure to thank Claire Voisin for her help and support, as well as many enlightening discussions during the writing of this paper. I am grateful to Eyal Markman for kindly explaining me the results of \cite{Ma10}. I would also like to thank Daniel Huybrechts for pointing out the relevance of Verbitsky's results and for the interesting discussions we had around the manuscript during his nice invitation to the university of Bonn, as well as Burt Totaro and the referee for many useful comments.

\section{General remarks on the Lefschetz standard conjecture}

This section is devoted to some general remarks on the Lefschetz standard conjecture. Although some are well-known to specialists, we include them here for ease of reference. Let us first recall the statement of the conjecture.

 Let $X$ be a smooth projective variety of dimension $n$ over $\C$. Let $\xi \in H^2(X, \Q)$ be the cohomology class of a hyperplane section of $X$. According to the hard Lefschetz theorem, see for instance \cite{Vo02}, Chapter 13, for all $k\in\{0, \ldots, n\}$, cup-product with $\xi^{n-k}$ induces an isomorphism
$$\cup \xi^{n-k} : H^{k}(X, \Q)\ra H^{2n-k}(X, \Q).$$

The Lefschetz standard conjecture was first stated in \cite{Gr69}, conjecture $B(X)$. It is the following.

\begin{conj}
Let $X$ and $\xi$ be as above. Then for all $k\in\{0, \ldots, n\}$, there exists an algebraic cycle $Z$ of codimension $k$ in the product $X\times X$ such that the correspondence
$$[Z]_* : H^{2n-k}(X, \Q)\ra H^{k}(X, \Q)$$
is the inverse of $\cup \xi^{n-k}$.
\end{conj}
If this conjecture holds for some specific $k$ on $X$, we will say the Lefschetz conjecture holds in degree $k$ for the variety $X$.

Let us recall the following easy lemma, see \cite{Kl91}, Theorem 4.1, which shows in particular that the Lefschetz conjecture does not depend on the choice of a polarization.

\begin{lem}
 Let $X$ and $\xi$ be as above. Then the Lefschetz conjecture holds in degree $k$ for $X$ if and only if there exists an algebraic cycle $Z$ of codimension $k$ in the product $X\times X$ such that the correspondence
$$[Z]_* : H^{2n-k}(X, \Q)\ra H^{k}(X, \Q)$$
is bijective.
\end{lem}

\begin{proof}
 Let $Z$ be as in the lemma. The morphism
$$[Z]_* \circ (\cup \xi^{n-k} \circ [Z]_*)^{-1}: H^{2n-k}(X, \Q)\ra H^{k}(X, \Q)$$
is the inverse of $\cup \xi^{n-k} : H^{k}(X, \Q)\ra H^{2n-k}(X, \Q).$ Now by the Cayley-Hamilton theorem, the automorphism $(\cup \xi^{n-k} \circ [Z]_*)^{-1}$ of $H^{2n-k}(X, \Q)$ is a polynomial in $(\cup \xi^{n-k} \circ [Z]_*)$. As such, it is defined by an algebraic correspondence. By composition, the morphism $[Z]_* \circ (\cup \xi^{n-k} \circ [Z]_*)^{-1}$ is defined by an algebraic correspondence, which concludes the proof.
\end{proof}

For the next results, we will need to work with primitive cohomology classes. Let us recall some notation. Let $S$ be a smooth polarized projective variety of dimension $l$. Let $L$ denote cup-product with the cohomology class of a hyperplane section. For any integer $k$ in $\{0, \ldots, l\}$, let $H^k(S, \Q)_{prim}$ denote the primitive part of $H^k(S, \Q)$, that is, the kernel of
$$L^{l-k+1} : H^k(S, \Q)\ra H^{2l-k+2}(S, \Q).$$
The cohomology groups of $S$ in degrees less than $l$ then admit a Lefschetz decomposition
$$H^k(S, \Q) = \bigoplus_{i\geq 0} L^i H^{k-2i}(S, \Q)_{prim}.$$
The following lemma is well-known, but we include it here for ease of reference as well as to keep track of the degrees for which we have to use the Lefschetz standard conjecture.



\begin{lem}\label{proj}
 Let $k$ be an integer, and let $S$ be a smooth projective scheme of dimension $l\geq k$. Consider the Lefschetz decomposition
$$H^k(S, \Q)=\bigoplus_{i\geq 0} L^i H^{k-2i}(S, \Q)_{prim},$$
where $L$ is the cup-product by the class of a hyperplane section. Assume that the Leschetz conjecture holds for $S$ in degrees up to $k-2$. Then the projections $H^k(S, \Q)\ra L^i H^{k-2i}(S, \Q)_{prim}$ are induced by algebraic correspondences.
\end{lem}

\begin{proof}
 By induction, it is enough to prove that the projection $H^k(S, \Q)\ra L H^{k-2}(S, \Q)$ is induced by an algebraic correspondence. Let $Z\subset S\times S$ be an algebraic cycle such that

$$[Z]_* : H^{2l-k+2}(S, \Q)\ra H^{k-2}(S, \Q)$$
is the inverse of $L^{l-k+2}$. Then the composition $ L\circ [Z]_*\circ L^{l-k+1}$ is the desired projection since $H^k(S, \Q)_{prim}$ is the kernel of $L^{l-k+1}$ in $H^k(S, \Q)$ .
\end{proof}

\bigskip

The next result is the starting point of our paper. It shows that the Lefschetz standard conjecture in degree $k$ on $X$ is equivalent to the existence of a sufficiently big family of codimension $k$ algebraic cycles in $X$, and allows us to work on the product of $X$ with any variety.

\begin{prop}\label{surj}
 Let $X$ be a smooth projective variety of dimension $n$, and let $k\leq n$ be an integer. Then the Lefschetz conjecture is true in degree $k$ for $X$ if and only if there exists a smooth projective scheme $S$ of dimension $l\geq k$ satisfying the Lefschetz conjecture in degrees up to $k-2$ and a codimension $k$ cycle $Z$ in $X\times S$ such that the morphism
$$[Z]_* : H^{2l-k}(S, \Q)\ra H^k(X, \Q)$$
induced by the correspondence $Z$ is surjective.
\end{prop}

\begin{proof}
Taking $S=X$, the ''only if'' part is obvious. For the other statement, fix a polarization on $S$, and let $L$ be the cup-product with the class of a hyperplane section of $S$. Consider the morphism $s : H^k(S, \Q) \ra H^k(S, \Q)$ which is given by multiplication by $(-1)^{i}$ on $L^i H^{k-2i}(S, \Q)_{prim}$. By the Hodge index theorem, the pairing
$$H^k(S, \C)\otimes H^k(S, \C)\ra \C, \,\alpha\otimes\beta \mapsto \int_S \alpha \cup L^{l-k}(s(\beta))$$
turns $H^k(S, \Q)$ into a polarized Hodge structure. Furthermore, Lemma \ref{proj} shows that $s$ is induced by an algebraic correspondence.

 We have a morphism $[Z]_* : H^{2l-k}(S, \Q)\ra H^k(X, \Q)$ which is surjective. Its dual $[Z]^* : H^{2n-k}(X,\Q)\ra H^k(S, \Q)$ is injective, where $n$ is the dimension of $X$. Let us consider the composition
$$[Z]_*\circ L^{l-k} \circ s \circ [Z]^* : H^{2n-k}(X, \Q)\ra H^k(X, \Q),$$
It is defined by an algebraic correspondence, and it is enough to show that it is a bijection. Since $H^{2n-k}(X, \Q)$ and $H^k(X, \Q)$ have the same dimension, we only have to prove it is injective.

Let $\alpha\in H^{2n-k}(X,\Q)$ lie in the kernel of the composition. For any $\beta\in H^{2n-k}(X, \Q)$, we get
$$([Z]^*\beta)\cup ((L^{l-k} \circ s)([Z]^*\alpha))=0.$$
Since $[Z]^*(H^{2n-k}(X, \Q))$ is a sub-Hodge structure of the polarized Hodge structure $H^k(S, \Q)$, the restriction of the polarization $$<u, v>=\int_S u \cup (L^{l-k}\circ s)(v)$$
on $H^k(S, \Q)$ to this subspace is nondegenerate, which shows that $\alpha$ is zero.
\end{proof}

\begin{rk}
 Using the weak Lefschetz theorem, one can always reduce to the case where $S$ is of dimension $k$.
\end{rk}

\begin{cor}\label{tr}
 Let $X$ be a smooth projective variety of dimension $n$, and let $k\leq n$ be an integer. Assume the Lefschetz conjecture for all varieties in degrees up to $k-2$ and that the generalized Hodge conjecture is true for $H^k(X,\Q)$.

Then the Lefschetz conjecture is true in degree $k$ for $X$ if and only if there exists a smooth projective scheme $S$, of dimension $l$, and a codimension $k$ cycle $Z$ in $CH^k(X\times S)$ such that the morphism
\begin{equation}\label{Lef}
H^{l}(S, \Omega^{l-k}_S)\ra H^k(X, \mathcal{O}_X)
\end{equation}
induced by the morphism of Hodge structures
$$[Z]_* : H^{2l-k}(S, \C)\ra H^k(X, \C)$$
is surjective.
\end{cor}

\begin{rk}\label{incond}
 Note that this corollary is unconditional for $k=2$ since the generalized Hodge conjecture is just the Hodge conjecture for divisors, and the Lefschetz standard conjecture is obvious in degree $0$.
\end{rk}

\begin{proof}
Let $X$, $S$ and $Z$ be as in the statement of the corollary. Let $H$ be the image of $H^{2l-k}(S, \Q)$ by $[Z]_*$. By (\ref{Lef}), we have $H^{k, 0}=H^k(X, \mathcal{O}_X)$. Let $H'$ be a sub-Hodge structure of $H^k(X,\Q)$ such that $H^k(X,\Q)=H\oplus H'$. Then $H'^{k,0}=0$. As $H'$ has no part of type $(k,0)$, the generalized Hodge conjecture then predicts that there exists a smooth projective variety $X'$ of dimension $n-1$, together with a proper morphism $f: X'\ra X$ such that $H'$ is contained in $f_* H^{k-2}(X',\Q)$.

If the Lefschetz conjecture is true in degree $k-2$, then it is true for $H^{k-2}(X',\Q)$. As a consequence, we get a cycle $Z'$ of codimension $k-2$ in $X'\times X'$ such that $[Z']_* : H^{2(n-1)-k+2}(X', \Q)\ra H^{k-2}(X', \Q)$ is surjective. Consider the composition
$$H^{2(n-1)+2-k}(X'\times \mathbb P^1, \Q) \twoheadrightarrow H^{2(n-1)-k+2}(X', \Q)\twoheadrightarrow H^{k-2}(X', \Q) \ra H^k(X, \Q),$$
the first map being the pullback by any of the immersions $X'\ra X'\times \mathbb P^1, x'\mapsto (x', x)$, the second one being $[Z']_*$ and the last one $f_*$. This composition is induced by an algebraic correspondence $Z''\hookrightarrow X'\times \mathbb P^1\times X$, and is surjective onto $f_* H^{k-2}(X',\Q)$. It is easy to assume, after taking products with projective spaces, that $S$ and $X'\times \mathbb P^1$ have the same dimension. Now since the subspaces $H$ and
$f_* H^{k-2}(X',\Q)$ generate $H^k(X,\Q)$, the correspondence induced by the cycle $Z+Z''$ in $(S\coprod (X'\times \mathbb P^1))\times X$ satisfies the hypotheses of Proposition \ref{surj}.
\end{proof}

\bigskip

With the notations of the previous corollary, in case, $Z$ is flat over $X$, we have a family of codimension $k$ algebraic cycles in $X$ parametrized by $S$. The next theorem shows that the map (\ref{Lef}), which is the one we have to study in order to prove the Lefschetz conjecture in degree $k$ for $X$, does not depend on the global geometry of $S$, and can be computed locally on $S$. This will allow us to give an explicit description of the map (\ref{Lef}) in terms of the deformation theory of the family $Z$ in the next section.

Let us first recall a general cohomological invariant for families of algebraic cycles. We follow \cite{Vo02}, 19.2.2, see also \cite{IVHS3}, \cite{Vo88} for related discussions. In the previous setting, $Z$, $X$ and $S$ being as before, the algebraic cycle $Z$ has a class
$$[Z]\in H^k(X\times S, \Omega^k_{X\times S}).$$
Using the K\"unneth formula, this last group maps to
$$H^0(S, \Omega^k_S)\otimes H^k(X, \mathcal{O}_X), $$
which means that the cohomology class $[Z]$ gives rise to a morphism of sheaves on $S$
\begin{equation}\label{def}
\phi_Z : \bigwedge^k \mathcal T_S \ra H^k(X, \mathcal O_X)\otimes \mathcal O_S,
\end{equation}
where $\mathcal{T}_S$ is the tangent sheaf of $S$.
If $s$ is a complex point of $S$, let $\phi_{Z,s}$ be the morphism $\bigwedge^k \mathcal T_{S,s} \ra H^k(X, \mathcal O_X)$ coming from $\phi_Z$. 

\bigskip

Note that the definition of $\phi_{Z,s}$ is local on $S$. Indeed, the map $H^k(X\times S, \Omega^k_{X\times S})\ra H^0(S, \Omega^k_S)\otimes H^k(X, \mathcal{O}_X)$ factors through the restriction map 
$$H^k(X\times S, \Omega^k_{X\times S})\ra H^0(S, R^k p_* \Omega^k_{X\times S}),$$ 
where $p$ is the projection from $X\times S$ to $S$, corresponding to the restriction of a cohomology class to the fibers of $p$. Actually, it can be shown that it only depends on the first order deformation $Z_s^{\epsilon}$ of $Z_s$ in $X$, see \cite{Vo02}, Remarque 19.12 under rather weak assumptions. We will recover this result in the next section by giving an explicit formula for $\phi_{Z,s}$. This fact is the one that allows us to reduce the Lefschetz standard conjecture to a variational statement.

The next theorem shows, using the map $\phi_{Z,s}$, that the Lefschetz conjecture can be reduced to the existence of local deformations of algebraic cycles in $X$.

\begin{thm}\label{vari}
 Let $X$ be a smooth projective variety. Assume as in Corollary \ref{tr} that the generalized Hodge conjecture is true for $H^k(X,\Q)$ and the Lefschetz conjecture holds for smooth projective varieties in degree $k-2$.

Then the Lefschetz conjecture is true in degree $k$ for $X$ if and only if there exist a smooth quasi-projective scheme $S$, a codimension $k$ cycle $Z$ in $CH^k(X\times S)$ and a point $s\in S$ such that the morphism
\begin{equation}
\phi_{Z,s} : \bigwedge^k \mathcal T_{S,s} \ra H^k(X, \mathcal O_X)
\end{equation}
is surjective.
\end{thm}

\begin{proof}
 Assume the hypothesis of the theorem holds. Up to taking a smooth projective compactification of $S$ and taking the adherence of $Z$, we can assume $S$ is smooth projective. The morphism of sheaves
$$\phi_Z : \bigwedge^k \mathcal T_S \ra H^k(X, \mathcal O_X)\otimes \mathcal O_S$$
that we constructed earlier corresponds to an element of the group
$$Hom_{\mathcal O_S}(\bigwedge^k \mathcal T_S, H^k(X, \mathcal O_X)\otimes \mathcal O_S)=H^0(\Omega^k_S\otimes H^k(X, \mathcal{O}_X)),$$
which in turn using Serre duality corresponds to a morphism
$$H^l(S, \Omega^{l-k}_S)\ra H^k(X, \mathcal{O}_X),$$
where $l$ is the dimension of $S$.

By the definition of $\phi_Z$, this morphism is actually the morphism (\ref{Lef}) of Corollary \ref{tr}. Indeed, this last morphism was constructed using the K\"unneth formula for $X\times S$, Poincar\'e duality and taking components of the Hodge decomposition, which is the way $\phi_Z$ is defined, since Serre duality is compatible with Poincar\'e duality.

Moreover, by construction, if $\phi_{Z,s}$ is surjective, then $H^l(S, \Omega^{l-k}_S)\ra H^k(X, \mathcal{O}_X)$ is. As for the converse, if $H^l(S, \Omega^{l-k}_S)\ra H^k(X, \mathcal{O}_X)$ is surjective, then we can find points $s_1, \ldots, s_r$ of $s$ such that the images of the $\phi_{Z,s_i}$ generate $H^k(X, \mathcal{O}_X)$. Replacing $S$ by $S^r$, the cycle $Z$ by the disjoint union of the $Z_i=p_i^* Z$, where $p_i : S^r\times X \ra S\times X$ is the projection on the product of the $i$-th factor, and $s$ by $(s_1, \ldots, s_r)$, this concludes the proof by Corollary \ref{tr}.
\end{proof}

The important part of this theorem is that it does not depend on the global geometry of $S$, but only on the local variation of the family $Z$. As such, it makes it possible to use deformation theory and moduli spaces to study the Lefschetz conjecture, especially in degree $2$ where Theorem \ref{vari} is unconditional by Remark \ref{incond}.

\section{A local computation}

Let $X$ be a smooth variety and $S$ a smooth scheme, $X$ being projective and $S$ quasi-projective. Let $Z$ be a cycle of codimension $k$ in the product $X\times S$. As we saw earlier, for any point $s\in S$, the correspondence defined by $Z$ induces a map
$$\phi_{Z,s} : \bigwedge^k \mathcal T_{S,s} \ra H^k(X, \mathcal O_X)$$
The goal of this section is to compute this map in terms of the deformation theory of the family $Z$ of cycles on $X$ parametrized by $S$. We will formulate this result when the class of $Z$ in the Chow group of $X\times S$ is given by the codimension $k$ part $ch_k(\mathcal E)$ of the Chern character of a vector bundle $\mathcal E$ over $X\times S$. It is well-known that we obtain all the rational equivalence classes of algebraic cycles as linear combinations of those.

\bigskip

Let us now recall general facts about the deformation theory of vector bundles and their Atiyah class.
Given a vector bundle $\mathcal{E}$ over $X\times S$, and $p$ being the projection of $X\times S$ to $S$, let $\mathcal{E}xt^1_p(\mathcal{E}, \mathcal{E})$ be the sheafification of the presheaf $U\mapsto \mathrm{Ext}^1_{\mathcal{O}_{X\times U}}(\mathcal{E}_{|X\times U}, \mathcal{E}_{|X\times U})$ on $S$. The deformation of vector bundles determined by $\mathcal E$ is described by the Kodaira-Spencer map. This is a map of sheaves
$$\rho : \mathcal{T}_S\ra \mathcal{E}xt^1_p(\mathcal{E}, \mathcal{E}),$$
where $\mathcal{T}_S$ is the tangent sheaf to $S$. Let $s$ be a complex point of $S$. The Kodaira-Spencer map at $s$ is given by the composition
$$\rho_s : T_{S, s} \ra \mathcal{E}xt^1_p(\mathcal{E}, \mathcal{E})_s\ra \mathrm{Ext}^1(\mathcal{E}_s, \mathcal{E}_s), $$
the last one being the canonical one.

In the next section, we will use results of Verbitsky which allow us to produce unobstructed elements of $\mathrm{Ext}^1(\mathcal{E}_s, \mathcal{E}_s)$ in the hyperholomorphic setting.

\bigskip

Associated to $\mathcal E$ as well are the images in $H^k(X\times S, \Omega^k_{X\times S})$ of the Chern classes of $\mathcal{E}$, which we will denote by $c_k(\mathcal{E})$ with a slight abuse of notation. We also have the  images $ch_k(\mathcal{E})\in H^k(X\times S, \Omega^k_{X\times S})$ of the Chern character.

\bigskip

The link between Chern classes and the Kodaira-Spencer map is given by the Atiyah class. It is well-known that the Chern classes of $\mathcal F$ can be computed from its Atiyah class $A(\mathcal{F})\in \mathrm{Ext}^1(\mathcal{F}, \mathcal F\otimes \Omega^1_Y)$, see \cite{At57}, \cite{HL97}, Chapter 10 :

\begin{prop}\label{chern}
 For $k$ a positive integer, let $\alpha_k\in H^k(Y, \Omega_Y^k)$ be the trace of the element $A(\mathcal F)^k\in \mathrm{Ext}^k(\mathcal{F}, \mathcal F\otimes \Omega^k_Y)$ by the trace map. Then
$$\alpha_k=k!\,ch_k(\mathcal F).$$
\end{prop}

Now in the relative situation with our previous notation, the vector bundle $\mathcal E$ has an Atiyah class $A(\mathcal E)$ with value in $\mathrm{Ext}^1(\mathcal{E}, \mathcal E\otimes \Omega^1_{X\times S})$. The latter group maps to the group $H^0(S, \mathcal{E}xt^1_p(\mathcal E, \mathcal E\otimes \Omega^1_{X\times S}))$, which contains
$$H^0(S, \mathcal{E}xt^1_p(\mathcal E, \mathcal E)\otimes \Omega^1_{S})=\mathrm{Hom}(\mathcal{T}_S\ra \mathcal{E}xt^1_p(\mathcal{E}, \mathcal{E}))$$
as a direct factor. We thus get a morphism of sheaves
$$\tau : \mathcal{T}_S\ra \mathcal{E}xt^1_p(\mathcal{E}, \mathcal{E}).$$
For the following well-known computation, see \cite{HL97} or \cite{Il71}, Chapter IV.

\begin{prop}\label{KS}
 The map $\tau$ induced by the Atiyah class of $\mathcal E$ is equal to the Kodaira-Spencer map $\rho$.
\end{prop}

Those two results make it possible to give an explicit description of the map $\phi_{Z}$ of last section in case the image of $Z$ in the Chow group of $X\times S$ is given by the codimension $k$ part $ch_k(\mathcal E)$ of the Chern character of a vector bundle $\mathcal E$ over $X\times S$. First introduce a map of sheaves coming from the Kodaira-Spencer map.

For $k$ a positive integer, let
$$\phi_k(\mathcal E) : \bigwedge^k \mathcal T_S \ra H^k(X, \mathcal O_X)\otimes \mathcal O_S$$
be the composition of the $k$-th alternate product of the Kodaira-Spencer map with the map
$$\bigwedge^k \mathcal{E}xt^1_p(\mathcal{E}, \mathcal{E})\ra \mathcal Ext^k_p(\mathcal E, \mathcal E)\ra H^k(X, \mathcal O_X)\otimes \mathcal O_S,$$
the first arrow being Yoneda product and the second being the trace map.

\begin{lem}
We have
$$\phi_k(\mathcal E)=k!\,\phi_{ch_k(\mathcal E)},$$
where $\phi_{ch_k(\mathcal E)}$ is the map appearing in (\ref{def}).
\end{lem}

\begin{proof}
We have the following commutative diagram
\begin{center}
$$\xymatrix{\mathrm{Ext}^1(\mathcal E, \mathcal{E}\otimes \Omega^1_{X\times S})^{\otimes k} \ar[r]\ar[d] & \mathrm{Ext}^k(\mathcal E, \mathcal{E}\otimes \Omega^k_{X\times S})\ar[r]\ar[d] & H^k(X\times S, \Omega^k_{X\times S})\ar[d]\\
H^0(S, \mathcal{E}xt^1_p(\mathcal E, \mathcal{E}\otimes \Omega^1_{X\times S}))^{\otimes k} \ar[r]\ar[d] & H^0(S, \mathcal{E}xt^k_p(\mathcal E, \mathcal{E}\otimes \Omega^k_{X\times S}))\ar[r]\ar[d] & H^0(S, R^k p_*\Omega^k_{X\times S})\ar[d]\\
H^0(S, \Omega^1_S\otimes \mathcal{E}xt^1_p(\mathcal E, \mathcal{E}))^{\otimes k} \ar[r] & H^0(S, \Omega^k_S\otimes \mathcal{E}xt^k_p(\mathcal E, \mathcal{E}))\ar[r] & H^0(S, \Omega^k_S\otimes H^k(X, \mathcal{O}_X)),}$$
\end{center}
where the horizontal maps on the left are given by Yoneda product, the horizontal maps on the right side are the trace maps, the upper vertical maps come from the Leray exact sequence associated to $p$, and the lower vertical maps come from the projection $\Omega^1_{X\times S} \ra p^* \Omega^1_S$.

By definition, and using Proposition \ref{KS}, the element $A(\mathcal E)^{\otimes k}\in \mathrm{Ext}^1(\mathcal E, \mathcal{E}\otimes \Omega^1_{X\times S})^{\otimes k}$ maps to $$\phi_k(\mathcal E)\in \mathrm{Hom}(\bigwedge^k \mathcal T_S, H^k(X, \mathcal O_X)\otimes \mathcal O_S)=H^0(S, \Omega^k_S\otimes H^k(X, \mathcal{O}_X)),$$
following the left side, then the lower side of the diagram.
On the other hand, Proposition \ref{chern} shows that it also maps to $k!\,\phi_{ch_k(\mathcal E)}$, following the upper side, then the right side of the diagram. This concludes the proof.
\end{proof}

As an immediate consequence, we get the following criterion.

\begin{thm}\label{crit}
Let $X$ be a smooth projective variety, and assume the same hypotheses as in Theorem \ref{vari}. Then the Lefschetz conjecture is true in degree $k$ for $X$ if there exists a smooth quasi-projective scheme $S$, a vector bundle $\mathcal E$ over $X\times S$, and a point $s\in S$ such that the morphism
\begin{equation}\label{critere}
\phi_k(\mathcal E)_s : \bigwedge^k \mathcal T_{S,s} \ra H^k(X, \mathcal O_X)
\end{equation}
induced by $\phi_k(\mathcal E)$ is surjective.
\end{thm}

\begin{rk}
Since Chern classes of vector bundles generate the Chow groups of smooth varieties, we can get a converse to the preceding statement by stating the theorem for complexes of vector bundles -- or of coherent sheaves. The statement would be entirely similar. As we will not use it in that form, we keep the preceding formulation.
\end{rk}

\paragraph{Example.} Let $A$ be a polarized complex abelian variety of dimension $g$. The tangent bundle of $A$ is canonically isomorphic to $H^1(A, \mathcal{O}_A)\otimes \mathcal{O}_A$. The trivial line bundle $\mathcal{O}_A$ on $A$ admits a family of deformations parametrized by $A$ itself such that the Kodaira-Spencer map $T_{A,O}\ra H^1(A, \mathcal{O}_A)$ is the identity under the above identification. Now the induced deformation of $\mathcal{O}_A\oplus \mathcal{O}_A$ parametrized by $A\times A$ satisfies the criterion of Theorem \ref{crit}, since the map $\bigwedge^2 H^1(A, \mathcal{O}_A) \ra H^2(A, \mathcal{O}_A)$ given by cup-product is surjective and identifies with the map (\ref{critere}). Of course, the Lefschetz conjecture for abelian varieties is well-known, see \cite{Li68}, Theorem 3.

\section{The case of hyperk\"{a}hler varieties}

In this section, we describe how Verbitsky's theory of hyperholomorphic bundles on hyperk\"{a}hler varieties as developed in \cite{Ver96} and \cite{Ver99} makes those a promising source of examples for theorem \ref{crit}. Unfortunately, we were not able to provide examples, as it appears some computations of dimensions of moduli spaces in \cite{Ver99} were incorrect, but we will show how the existence of nontrivial examples of moduli spaces of hyperholomorphic bundles on hyperk\"{a}hler varieties as conjectured in \cite{Ver99} implies the Lefschetz standard conjecture in degree $2$.

\subsection{Hyperholomorphic bundles on hyperk\"ahler varieties}

See \cite{Be83} for general definitions and results. An irreducible hyperk\"ahler variety is a simply connected k\"ahler manifold which admits a closed everywhere non-degenerate two-form which is unique up to a factor. As such, an irreducible hyperk\"ahler variety $X$ has $H^{2,0}(X, \mathcal{O}_X)=\C$, and Theorem \ref{crit} takes the following simpler form in degree $2$.

\begin{thm}
 Let $X$ be an irreducible projective hyperk\"ahler variety. The Lefschetz conjecture is true in degree $2$ for $X$ if there exists a smooth quasi-projective variety $S$, a vector bundle $\mathcal E$ over $X\times S$, and a point $s\in S$ such that the morphism
\begin{equation}
\phi_2(\mathcal E)_s : \bigwedge^2 \mathcal T_{S,s} \ra H^2(X, \mathcal O_X),
\end{equation}
induced by the Kodaira-Spencer map and the trace map, is nonzero.
\end{thm}

\bigskip

In the paper \cite{Be83}, Beauville constructs two families of projective irreducible hyperk\"ahler varieties in dimension $2n$ for every integer $n$. Those are the $n$-th punctual Hilbert scheme $S^{[n]}$ of a projective $K3$ surface $S$ and the generalized Kummer variety $K_n$ which is the fiber at the origin of the Albanese map from $A^{[n+1]}$ to $A$, where $A$ is an abelian surface and $A^{[n+1]}$ is the $n+1$-st punctual Hilbert scheme of $A$.

The Bogomolov-Tian-Todorov theorem, see \cite{Bo78}, \cite{Ti87}, \cite{To89}, states that the local moduli space of deformations of an irreducible hyperk\"ahler variety is unobstructed. Small deformations of a hyperk\"ahler variety remain hyperk\"ahler, and in the local moduli space of $S^{[n]}$ and $K_n$, the projective hyperk\"ahler varieties form a dense countable union of hypersurfaces. The varieties $S^{[n]}$ and $K_n$ have Picard number at least $2$, whereas a very general projective irreducible hyperk\"ahler variety has Picard number $1$, hence is not of this form. Except in dimension $6$ and $10$, where O'Grady constructs in \cite{OG99} and \cite{OG03} new examples, all the known hyperk\"ahler varieties are deformations of $S^{[n]}$ or $K_n$.

\bigskip

The Lefschetz standard conjecture is easy to prove in degree $2$ for $S^{[n]}$ (resp. $K_n$), using the tautological correspondence with the $K3$ surface (resp. the abelian surface), see \cite{Ar06}, Corollary 7.5. In terms of Theorem \ref{crit}, one can show that the tautological sheaf on $S^{[n]}$ (resp. $K_n$) associated to the tangent sheaf of $S$ has enough deformations to prove the Lefschetz conjecture in degree $2$. Since the tautological correspondence between $S$ and $S^{[n]}$ gives an isomorphism between $H^{2,0}(S)$ and $H^{2,0}(S^{[n]})$, checking that the criterion is satisfied amounts to the following.

\begin{prop}
 Let $S$ be a projective $K3$ surface. Then there exists a smooth quasi-projective variety $M$ with a distinguished point $O$ parametrizing deformations of $\mathcal{T}_S$ and a vector bundle $\mathcal E$ over $M\times M$ such that $\mathcal{E}_{|\{O\times S\}} \simeq \mathcal{T}_S$, such that the map
$$\phi_2(\mathcal E)_O : \bigwedge^2 \mathcal T_{M,O} \ra H^2(S, \mathcal O_S)$$
induced by the Kodaira-Spencer map and the trace map, is nonzero.
\end{prop}

\begin{proof}
This is proved by Mukai in \cite{Mu84}. A Riemann-Roch computation proves that the moduli space of deformations of of the tangent bundle of a $K3$ surface is smooth of dimension $90$.
\end{proof}

This last proof is of course very specific to Hilbert schemes and does not apply as such to other hyperk\"ahler varieties. We feel nonetheless that it exhibits general facts about hyperk\"ahler varieties which seem to give strong evidence to the Lefschetz conjecture in degree $2$.


\subsection{Consequences of the existence of a hyperk\"ahler structure on the moduli space of stable hyperholomorphic bundles}

In his paper \cite{Mu84}, Mukai studies the moduli spaces of some stable vector bundles on K3 surfaces and endows them with a symplectic structure by showing that the holomorphic two-form induced by (\ref{critere}) on the moduli space is nondegenerate. Of course, this result is not directly useful when dealing with the Lefschetz standard conjecture in degree 2 as it is trivial for surfaces. Nevertheless, Verbitsky shows in \cite{Ver96} that it is possible to extend Mukai's result to the case of higher-dimensional hyperk\"ahler varieties.

\bigskip

Before describing Verbitsky's results, let us recall some general facts from linear algebra around quaternionic actions and symplectic forms. This is all well-known, and described for instance in \cite{Be83}, Example 3, and \cite{Ver96}, section 6. Let $\mathbb H$ denote the quaternions, and let $V$ be a real vector space endowed with an action of $\mathbb H$ and a euclidean metric $(,)$.

Let $I\in \mathbb H$ be a quaternion such that $I^2=-1$. The action of $I$ on $V$ gives a complex structure on $V$. We say that $V$ is quaternionic hermitian if the metric on $V$ is hermitian for all such complex structures $I$. Fix such an $I$, and choose $J$ and $K$ in $\mathbb H$ satisfying the quaternionic relations $I^2=J^2=K^2=-Id, IJ=-JI=K$. We can define on $V$ a real symplectic form $\eta$ such that $\eta(x,y)=(x, Jy)+i(x, Ky)$. This symplectic form does not depend on the choice of $J$ and $K$. Furthermore, $\eta$ is $\C$-bilinear for the complex structure induced by $I$. Now given such $I$ and $\eta$ on $V$, it is straightforward to reconstruct a quaternionic action on $V$ by taking the real and complex parts of $\eta$.

\smallskip

Taking $V$ to be the tangent space to a complex variety, we can globalize the previous computations to get the following. Let $X$ be an irreducible hyperk\"ahler variety with given K\"ahler class $\omega$. Then the manifold $X$ is endowed with a canonical hypercomplex structure, that is, three complex structures $I, J, K$ which satisfy the quaternionic relations $I^2=J^2=K^2=-Id, IJ=-JI=K$. It is indeed possible to check that $J$ and $K$ obtained as before are actually integrable. Conversely, the holomorphic symplectic form on $X$ can be recovered from $I, J, K$ and a K\"ahler form on $X$ with class $\omega$.

If $\mathcal E$ is a complex hermitian vector bundle on $X$ with a hermitian connection $\theta$, we say that $\mathcal E$ is hyperholomorphic if $\theta$ is compatible with the three complex structures $I, J$ and $K$. In case $\mathcal E$ is stable, this is equivalent to the first two Chern classes of $\mathcal E$ being Hodge classes for the Hodge structures induced by $I, J$ and $K$, see \cite{Ver96}, Theorem 2.5. This implies that any stable deformation of a stable hyperholomorphic bundle is hyperholomorphic.  It is a consequence of Yau's theorem, see \cite{Ya78} that the tangent bundle of $X$ is a stable hyperholomorphic bundle.

\bigskip

Let $\mathcal E$ be a stable hyperholomorphic vector bundle on $X$, and let $S=Spl(\mathcal E, X)$ be the reduction of the coarse moduli space of stable deformations of $\mathcal E$ on $X$. For $s$ a complex point of $S$, let $\mathcal E_s$ be the hyperholomorphic bundle corresponding to a complex point $s$ in $S$. The Zariski tangent space to $S$ at $s$ maps to $Ext^1(\mathcal{E}_s, \mathcal{E}_s)$ using the map from $S$ to the coarse moduli space of stable deformations of $\mathcal E$. We can now define a global section $\eta_S$ of $\mathcal Hom(\mathcal T_S\otimes \mathcal T_S, \mathcal O_S)$, where $\mathcal T_S$ is the tangent sheaf to $S$, by the composition
$$\mathcal T_{S,s}\otimes \mathcal T_{S,s} \ra \bigwedge^2 Ext^1(\mathcal{E}_s, \mathcal{E}_s)\ra Ext^2(\mathcal E_s, \mathcal E_s)\ra H^2(X, \mathcal O_X)=\C$$
as in the preceding section. The following is due to Verbitsky, see part (iv) of the proof in section 9 of \cite{Ver96} for the second statement.

\begin{thm}(\cite{Ver96}, Theorem 6.3)\label{mod}
Let $Spl(\mathcal E, X)$ be the reduction of the coarse moduli space of stable deformations of $\mathcal E$ on $X$. Then $S=Spl(\mathcal E, X)$ is endowed with a canonical hyperk\"ahler structure. The holomorphic section of $\mathcal Hom(\mathcal T_S\otimes \mathcal T_S, \mathcal O_S)$ induced by this hyperk\"ahler structure is $\eta_S$.
\end{thm}
In this theorem, $S$ does not have to be smooth. We use Verbitsky's definition of a singular hyperk\"ahler variety as in \cite{Ver96}, Definition 6.4.

\bigskip

We can now prove Theorem \ref{nd}.

\begin{proof}[\textbf{Proof of Theorem \ref{nd}}]
Let $X$ be a smooth projective irreducible hyperk\"ahler variety, and let $\mathcal E$ be a stable hyperholomorphic bundle on $X$. Assume that $\mathcal E$ has a nontrivial positive-dimensional family of deformations, and let $s$ be a smooth point of $S=Spl(\mathcal E, X)$ such that $\mathcal T_{S,s}$ is positive dimensional. We can choose a smooth quasi-projective variety $S'$ with a complex point $s'$ and a family $\mathcal E_{S'}$ of stable hyperholomorphic deformations of $\mathcal E$ on $X$ parametrized by $S'$ such that the moduli map $S'\ra S$ maps $s'$ to $s$ and is \'etale at $s'$. Since $\eta_S$ induces a symplectic form on $\mathcal T_{S,s}$, the map
$$\phi_2(\mathcal E_{S'})_s' : \bigwedge^2 \mathcal T_{S',s'} \ra H^2(X, \mathcal O_X)=\C$$
is surjective as it identifies with $\eta_{S,s}$ under the isomorphism $\mathcal T_{S',s'}\xrightarrow{\sim}\mathcal T_{S,s}$. The result now follows from Theorem \ref{crit}.
\end{proof}

\bigskip

In order to prove Theorem \ref{stdef}, we need to recall some well-known results on deformations of hyperk\"ahler varieties. Everything is contained in \cite{Be83}, Section 8 and \cite{Ver96}, Section 1. See also \cite{Hu99}, Section 1 for a similar discussion. Let $X$ be an irreducible hyperk\"ahler variety with given K\"ahler class $\omega$. Let $\eta$ be a holomorphic everywhere non-degenerate $2$-form on $X$. Let $q$ be the Beauville-Bogomolov quadratic form on $H^2(X, \Z)$, and consider the complex projective plane $P$ in $\mathbb{P}(H^2(X, \C))$ generated by $\eta, \overline{\eta}$ and $\omega$. There exists a quadric $Q$ of deformations of $X$ given the elements $\alpha\in P$ such that $q(\alpha)=0$ and $q(\alpha+\overline{\alpha})>0$.

Recalling that the tangent bundle of $X$ comes with an action of the groups of quaternions of norm $1$ given by the three complex structures $I, J, K$, which satisfy the quaternionic relations $I^2=J^2=K^2=-Id, IJ=-JI=K$, this quadric $Q$ of deformations of $X$ corresponds to the complex structures on $X$ of the form $aI+bJ+cK$ with $a,b,c$ being three real numbers such that $a^2+b^2+c^2=1$ -- those complex structures are always integrable. The quadric $Q$ is called a twistor line.

In this setting, let $d$ be the cohomology class of a divisor in $H^2(X, \C)$, and let $\alpha$ be in $Q$. This corresponds to a deformation $X_{\alpha}$ of $X$. The cohomology class $d$ corresponds to a rational cohomology class in $H^2(X_{\alpha}, \C)$, and it is the cohomology class of a divisor if and only if it is of type $(1,1)$, that is, if and only if $q(\alpha, d)=0$, where by $q$ we also denote the bilinear form induced by $q$. Indeed, $d$ is a real cohomology class, so if $q(\alpha, d)=0$, then $q(\overline{\alpha}, d)=0$ and $d$ is of type $(1,1)$. It follows from this computation that $d$ remains the class of a divisor for all the deformations of $X$ parametrized by $Q$ if and only if $q(\eta,d)=q(\omega, d)=0$.

\bigskip

We will work with the varieties $S^{[n]}$, the case of generalized Kummer varieties being completely similar. Let us start with a K3 surface $S$, projective or not, and let us consider the irreducible hyperk\"ahler variety $X=S^{[n]}$ given by the Douady space of $n$ points in $S$ -- this is K\"ahler by \cite{Va89}. In the moduli space $M$ of deformations of $X$, the varieties of the type $S'^{[n]}$ form a countable union of smooth hypersurfaces $H_i$. On the other hand, the hyperk\"ahler variety admits deformations parametrized by a twistor line, and those cannot be included in any of the $H_i$. Indeed, if that were the case, the class $e$ of the exceptional divisor of $X=S^{[n]}$ would remain algebraic in all the deformations parametrized by the twistor line. But this is impossible, as $e$ is a class of an effective divisor, which implies that $q(\omega, e)>0$, $\omega$ being a K\"ahler class, see \cite{Hu99}, 1.11 and 1.17.  

This computation actually shows that the twistor lines are transverse to the hypersurfaces $H_i$. Now the preceding defintion of the twistor line parametrizing deformations of an irreducible hyperk\"ahler $X$ shows that it moves continuously with deformations of $X$. Counting dimensions, this implies that the union of the twistor lines parametrizing deformations of Douady spaces of $n$ points on K3 surfaces cover a neighborhood of the $H_i$ in $M$. We thus get the following.

\begin{lem}\label{twist}
 Let $n$ be a positive integer, and let $X$ be a small projective deformation of the Douady space of $n$ points on a K3 surface. Then there exists a K3 surface $S$ and a twistor line $Q$ parametrizing deformations of $S^{[n]}$ such that $X$ is a deformation of $S^{[n]}$ along $Q$.
\end{lem}

\bigskip

The next result of Verbitsky is the main remaining ingredient we need to prove Theorem \ref{stdef}. Recall first that if $\mathcal E$ is a hyperholomorphic vector bundle on an irreducible hyperk\"ahler variety $X$, then by definition the bundle $\mathcal E$ deforms as $X$ deforms along the twistor line.

\begin{thm}\label{defmod}(\cite{Ver96}, Corollary 10.1)
Let $X$ be an irreducible hyperk\"ahler variety, and let $\mathcal E$ be a stable hyperholomorphic vector bundle on $X$, and let $Spl(\mathcal E, X)$ be the reduction of the coarse moduli space of stable deformations of $\mathcal E$ on $X$.

Then the canonical hyperk\"ahler structure on $Spl(\mathcal E, X)$ is such that if $Q$ is the twistor line parametrizing deformations of $X$, $Q$ is a twistor line parametrizing deformations of $Spl(\mathcal E, X)$ such that if $\alpha\in Q$, then $Spl(\mathcal E, X)_{\alpha}=Spl(\mathcal E_{\alpha}, X_{\alpha})$.
\end{thm}

This implies that the deformations of a hyperholomorphic bundle on $X$ actually deform as the complex structure of $X$ moves along a twistor line. We can now prove our last result.

\begin{proof}[\textbf{Proof of Theorem \ref{stdef}}]
Let $X$ be an irreducible projective hyperk\"ahler variety that is a deformation of the Douady space of $n$ points on some K3 surface. By a standard Hilbert scheme argument, in order to prove the Lefschetz conjecture for $X$, it is enough to prove it for an open set of the moduli space of projective deformations of $X$. By Lemma \ref{twist}, we can thus assume that $X$ is a deformation of some $S^{[n]}$ along a twistor line $Q$, where $S$ is a K3 surface. Let $\mathcal E$ on $S^{[n]}$ be a sheaf as in the statement of the theorem. By Theorems \ref{defmod} and \ref{nd}, we get a bundle $\mathcal E'$ which still satisfies the hypothesis of Theorem \ref{crit}. This concludes the proof.
\end{proof}

It is particularly tempting to use this theorem with the tangent bundle of $S^{[n]}$, which is stable by Yau's theorem and hyperholomorphic since its first two Chern classes are Hodge classes for all the complex structures induced by the hyperk\"ahler structure of $S^{[n]}$. Unfortunately, while Verbitsky announces in \cite{Ver99}, after the proof of Corollary 10.24, that those have some unobstructed deformations for $n=2$ and $n=3$,it seems that if $n=2$, the tangent bundle might be actually rigid. However, we get the following result by applying the last theorem to the tangent bundle.

\begin{cor}
Let $n$ be a positive integer. Assume that for every K3 surface $S$, the tangent bundle $\mathcal T_{S^{[n]}}$ of $ S^{[n]}$ has a nontrivial positive-dimensional family of deformations. Then the Lefschetz conjecture is true in degree $2$ for any projective deformation of the Douady space of $n$ points on a K3 surface.
\end{cor}

\begin{rk}
 The conditions of the corollary might be actually not so difficult to check. Indeed, Verbitsky's Theorem 6.2 of \cite{Ver96} which computes the obstruction to extending first-order deformations implies easily that the obstruction to deform $\mathcal T_{S^{[n]}}$ actually lies in $H^2(S^{[n]}, \Omega^2_{S^{[n]}})$, where we see this group as a subgroup of
$$\mathrm{Ext}^2(\mathcal T_{S^{[n]}}, \mathcal T_{S^{[n]}})=H^2(S^{[n]}, \Omega^{\otimes 2}_{S^{[n]}})$$
under the isomorphism $\mathcal T_{S^{[n]}}\simeq \Omega^1_{S^{[n]}}$.

Now the dimension of $H^2(S^{[n]}, \Omega^{2}_{S^{[n]}})$ does not depend on $n$ for large $n$, see for instance \cite{GS93}, Theorem 2. As a consequence, the hypothesis of the Corollary would be satisfied for large $n$ as soon as the dimension of $\mathrm{Ext}^1(\mathcal T_{S^{[n]}}, \mathcal T_{S^{[n]}})$ goes to infinity with $n$.
\end{rk}

\begin{rk}
Of course, our results might be apply to different sheaves. In the recent preprint \cite{Ma10}, Markman announces the construction of -- possibly twisted -- sheaves that, combined with our results, proves the Lefschetz standard conjecture in degree 2 for deformations of Hilbert schemes of K3 surfaces.
\end{rk}

\begin{rk}
 It is quite surprising that we make use of nonprojective K\"ahler varieties in these results dealing with the standard conjectures. Indeed, results like those of Voisin in \cite{Voi02} show that there can be very few algebraic cycles, whether coming from subvarieties or even from Chern classes of coherent sheaves, on general nonprojective K\"ahler varieties.
\end{rk}

\providecommand{\bysame}{\leavevmode ---\ }
\providecommand{\og}{``}
\providecommand{\fg}{''}
\providecommand{\smfandname}{et}
\providecommand{\smfedsname}{\'eds.}
\providecommand{\smfedname}{\'ed.}
\providecommand{\smfmastersthesisname}{M\'emoire}
\providecommand{\smfphdthesisname}{Th\`ese}

\end{document}